\documentclass[12pt]{amsart}

\newcommand{\side}{\te{side}}

\newcommand{\simw}{\sim_{w}}

\newcommand{\eluw}{\ell(u, w)}

\newcommand{\subs}{\sub S}

\newcommand{\iwj}{\ensuremath{\leftexp{I}W}^{J}}
\newcommand{\iW}{\ensuremath{\leftexp{I}W}}

\newcommand{\wi}{\ensuremath{W_{I}}}
\newcommand{\wj}{\ensuremath{W_{J}}}

\newcommand{\ins}{\in S}

\newcommand{\elx}{\ensuremath{\ell(x)}}
\newcommand{\ely}{\ensuremath{\ell(y)}}

\newcommand{\dlw}{\ensuremath{D_{L}(w)}}
\newcommand{\drw}{\ensuremath{D_{R}(w)}}

\newcommand{\lew}{\le w}

\newcommand{\elw}{\ensuremath{\ell(w)}}
\newcommand{\elu}{\ensuremath{\ell(u)}}

\newcommand{\elv}{\ensuremath{\ell(v)}}

\newcommand{\down}{\downarrow}

\newcommand{\inw}{\ensuremath{\in W}}

\newcommand{\wh}{\ensuremath{\widehat}}
\newcommand{\el}{\ensuremath{\ell}}

\DeclareFontFamily{U}{mathx}{\hyphenchar\font45}
\DeclareFontShape{U}{mathx}{m}{n}{
      <5> <6> <7> <8> <9> <10>
      <10.95> <12> <14.4> <17.28> <20.74> <24.88>
      mathx10
      }{}
\DeclareSymbolFont{mathx}{U}{mathx}{m}{n}
\DeclareFontSubstitution{U}{mathx}{m}{n}
\DeclareMathAccent{\widecheck}{0}{mathx}{"71}

\newcommand{\tw}{\textwidth}
\newcommand{\defeq}{\overset{\text{def}}{=}}

\renewcommand{\kill}[1]{}
\newcommand{\dummy}[1]{\mbox{}}

\makeatletter
\newcommand{\xequal}[2][]{\ext@arrow 0055{\equalfill@}{#1}{#2}}
\def\equalfill@{\arrowfill@\Relbar\Relbar\Relbar}
\makeatother

\newcommand{\Set}[2]{\ensuremath{\left\{{#1}\,\middle|\,{#2}\right\}}}
\newcommand{\ku}{\ensuremath{\emptyset}}

\renewcommand{\k}{\ensuremath{\ol{\mathrm{P}}}}

\newcommand{\h}{\hline}

\renewcommand{\k}[1]{\ensuremath{\left({#1}\right)}}

\newcommand{\ds}{\dots}

\newcommand{\bca}{\begin{cases}}
\newcommand{\eca}{\end{cases}}

\newcommand{\bpic}{\begin{picture}}\newcommand{\epic}{\end{picture}}

\newcommand{\beda}{\begin{edaenumerate}}
\newcommand{\eeda}{\end{edaenumerate}}

\newcommand{\cd}{\cdots}

\newcommand{\fn}{\footnotesize}

\newcommand{\q}{\quad}

\newcommand{\up}{\uparrow}

\newcommand{\bq}{\begin{quote}}\newcommand{\eq}{\end{quote}}

\newcommand{\ti}{\times}

\newcommand{\be}{\begin{enumerate}}\newcommand{\ee}{\end{enumerate}}
\newcommand{\bce}{\begin{center}}\newcommand{\ece}{\end{center}}
\newcommand{\bde}{\begin{description}}\newcommand{\ede}{\end{description}}
\newcommand{\bri}{\begin{flushright}}\newcommand{\eri}{\end{flushright}}
\newcommand{\bb}{\begin{block}}\newcommand{\eb}{\end{block}}
\newcommand{\bt}{\begin{thm}}\newcommand{\et}{\end{thm}}
\newcommand{\bpf}{\begin{proof}}\newcommand{\epf}{\end{proof}}
\newcommand{\bex}{\begin{ex}}\newcommand{\eex}{\end{ex}}
\newcommand{\bexr}{\begin{exr}}\newcommand{\eexr}{\end{exr}}
\newcommand{\bft}{\begin{fact}}\newcommand{\eft}{\end{fact}}
\newcommand{\brk}{\begin{rmk}}\newcommand{\erk}{\end{rmk}}
\newcommand{\ba}{\begin{align*}}\newcommand{\ea}{\end{align*}}
\newcommand{\bexe}{\begin{exe}}\newcommand{\eexe}{\end{exe}}
\newcommand{\tn}{\textnormal}
\newcommand{\rank}{\textnormal{rank\,}}
\newcommand{\bit}{\begin{itemize}}\newcommand{\eit}{\end{itemize}}

\newcommand{\bcm}{}

\newcommand{\hf}{\hfill}
\newcommand{\ci}{\CIRCLE}

\newcommand{\bd}{\begin{defn}}\newcommand{\ed}{\end{defn}}
\newcommand{\bp}{\begin{prop}}\newcommand{\ep}{\end{prop}}

\newcommand{\eh}{\emph}
\newcommand{\sub}{\subseteq}

\newcommand{\mb}{\mbox}
\newcommand{\te}{\text}\newcommand{\ph}{\phantom}

\newcommand{\then}{\Longrightarrow}
\newcommand{\leftexp}[2]{{\vphantom{#2}}^{#1}{#2}}

\renewcommand{\up}{\uparrow}

\renewcommand{\int}{\in T}

\renewcommand{\P}{\mathcal{P}}

\usepackage[dvipdfmx]{graphicx}
\usepackage[dvipsnames]{xcolor}
\usepackage{float,afterpage}
\usepackage{asymptote,layout}
\usepackage{wrapfig,epic}

\usepackage{geometry}
\usepackage{exscale,latexsym,bm}
\usepackage{amssymb,enumerate,amsmath,amsthm,amsfonts}
\usepackage{verbatim,fancybox,wasysym,fancyhdr,type1cm}
\usepackage[frame,all,poly,curve,knot,arrow]{xy}
\usepackage{colortbl}
\usepackage{boxedminipage}
\usepackage{multirow}

\graphicspath{{./figs/}}
\theoremstyle{definition}
\newtheorem{thm}{Theorem}[section]
\newtheorem{lem}[thm]{Lemma}
\newtheorem{prop}[thm]{Proposition}

\newtheorem{exr}[thm]{Exercise}
\newtheorem{ob}[thm]{Observation}

\newtheorem{ex}[thm]{Example}

\newtheorem{defn}[thm]{Definition}\newtheorem{rmk}[thm]{Remark}
\newtheorem{fact}[thm]{Fact}
\newtheorem{block}[thm]{}
\newtheorem*{exe}{Exercise}

\title[Bruhat graph 
structure of lower intervals]{
Bruhat order and graph structures of lower intervals
in Coxeter groups}
\author{Masato Kobayashi}
\date{\today}
\address{Department of Engineering\\
Kanagawa University, 3-27-1 Rokkaku-bashi, Yokohama 221-8686, Japan.}

\email{kobayashi@math.titech.ac.jp}
\begin{document}
\newcommand{\self}{\tn{side}_{w}}
\newcommand{\md}{\tn{mid}}
\newcommand{\ilj}{\ensuremath{\leftexp{I}{\el}^{J}}}
\newcommand{\ipj}{\ensuremath{\leftexp{I}{\mathstrut P}^{J}}}
\newcommand{\pd}{P_{\down}}
\begin{abstract} 
We show that any lower Bruhat interval in a Coxeter group is a disjoint union of certain two-sided cosets
as a consequence of Lifting Property and Subword Property.
Furthermore, we describe these details in terms of 
Bruhat graphs, graded posets, and two-sided quotients 
altogether.  For this purpose, we introduce some new ideas, quotient lower intervals and quotient Bruhat graphs.
\end{abstract}
\maketitle
\tableofcontents

\section{
Preliminaries}
Throughout this article, we denote by $W=(W, S, T, \el)$ a Coxeter system with $W$ the underlying Coxeter group and $S$ its Coxeter generators. 
By $T$ we mean the set of its reflections; $\el$ is the length function.
Unless otherwise noticed below, symbols $u, v, w, x, y$ are elements of $W$, $s\in S$, $t\in T$, $e$ is the group-theoretic unit of $W$ and $I, J$ are subsets of $S$.
See Appendix for more details on those definitions and notation.

\subsection{Bruhat graph}

By $u\to v$ we mean $\elu<\elv$ and $v=tu$ for some $t\in T$.
Define \eh{Bruhat order} $u\le w$ on $W$ if there is a directed path
\[
u=v_{0}\to v_{1}\to v_{2}\to \cd \to v_{k}=w.\]
For convenience, we use notation 
\[
\el(u, w)\defeq\el(w)-\elu\]
whenever $u\le w$.

\bd{The \emph{Bruhat graph} of $W$ is a directed graph for vertices $w\in W$ and for edges $u\to v$. 
}\ed 

For each subset $V\sub W$, we can also consider the induced subgraph with the vertex set $V$ (Bruhat subgraph).




\subsection{main results}

Any subset $V$ of $W$ in the form 
\[
[u, w]
\defeq
\{v\in W\mid u\le v\le w\} 
\]
is called a \eh{Bruhat interval}. 
In this article, we mainly study \eh{lower intervals}:
\[
B(w)
\defeq [e, w]=\{v\inw\mid v\lew\}.
\]
Although some results on such intervals are known in the books \cite{bb,billey}, there are still many unsolved problems.

Our main results (Theorems \ref{mth1} and \ref{mth2}) give a new description of poset structures $B(w)$ from aspects of Bruhat graphs and two-sided cosets/quotients.

\section{Quotient lower interval}

\subsection{two-sided cosets and quotients}

First, we discuss two-sided cosets and quotients, 
a little unfamiliar topic in the context of Coxeter groups; 
for one-sided cosets and quotients, see Appendix.




Let $V$ be a subset of $W$.
\begin{defn}
Say $V$ is a \eh{two-sided coset} if 
\[
V=W_{I}uW_{J}\defeq\{xuy\mid x\in \wi, y\in\wj\}\]
for some $u, I, J$.
\end{defn}

Every one-sided coset is a special case of two-sided cosets:
\[
W_{I}uW_{\ku}=W_{I}u  \,\text{   and   }\, W_{\ku}uW_{J}=uW_{J}.
\]
%

\begin{defn}[two-sided quotient]
$\iwj\defeq\iW\cap W^{J}$.
\end{defn}
Similarly, every one-sided quotient is a special case of 
two-sided quotients:
\[
{}^{\ku}W^{J}=W^{J} \,\text{   and   }\, {}^{I}W^{\ku}={}^{I}W.
\]

\begin{rmk}
Here, we would like to make sure some difference between one-sided 
cosets/quotients and two-sided ones.

\begin{itemize}
	\item Unlike one-sided cosets, 
	there is a choice of indices for an expression of a two-sided coset:
	it is possible that 
$W_{I}uW_{J}=W_{I'}uW_{J'}$ for some $(I, J)\ne (I', J')$.
	\item As analogy of Fact \ref{coset}, 
it is tempting to say that for each $v\in W_{I}uW_{J}$, 
there exists a unique pair $(x, y)\in\wi\ti \wj$ such that 
\begin{align*}
	v&=xuy,
	\\\el(v)&=\el(x)+\el(u)+\el(y).
\end{align*}
But this uniqueness is \eh{not} always true; it is possible that 
there exists a pair $(x', y')\in\wi\ti \wj$
such that 
\begin{align*}
	v&=x'uy',
	\\\el(v)&=\el(x')+\el(u)+\el(y')
\end{align*}
and moreover $(x', y')\ne (x, y)$.
However, it makes sense to speak of 
the minimal and maximal length coset representatives of $W_{I}uW_{J}$, say 
$v_{0}$ and $v_{1}$:
\[
	\el(v_{0})\le \elv\le \el (v_{1})
\]
for all $v\in W_{I}uW_{J}$.
Indeed, $W_{I}uW_{J}$ is a Bruhat interval $[v_{0}, v_{1}]$ and hence a 
graded poset of rank $\el(v_{0}, v_{1})$;
moreover, its Bruhat graph is $\el(v_{0}, v_{1})$-regular.

	\item 
Unlike one-sided quotients, a two-sided quotient $(\iwj, \le)$ is not necessarily graded with the rank function $\el$
(it may be graded with some other rank function, though);
that is, there may exist 
$u$, $w\in {}^{I}W^{J}$ such that 
\[
u< w,\q \el(u, w)\ge 2\]
and there does not exist $v\in {}^{I}W^{J}$ such that $u< v< w$.
\end{itemize}
\end{rmk}

%

\subsection{Bruhat cosets}

Let $B(w)\defeq [e, w]$, the lower interval determined by $w$.
Below, we study details on the graph structure of this interval in terms of our new idea, Bruhat cosets:

\bd{
A \emph{Bruhat coset} of $w$ is a two-sided coset of the form
$W_{D_L(w)}uW_{D_R(w)}$ for some $u\in B(w)$.
Denote this coset by $C_w(u)$ for simplicity.
}\ed
In particular, call $C_w(e)$ the \emph{bottom coset} and $C_w(w)$ the \emph{top coset}.

Several combinatorial properties of 
Bruhat cosets are in order.

\bp{We have $C_w(u)\subseteq B(w)$ for each $u\in B(w)$.
}\ep

\bpf{If $u\in B(w)$ and $s\in D_L(w)$, then $su$ stays in $B(w)$ because of Lifting Property (it does not matter whether $s\in D_L(u)$ or not).
As a consequence, if $x\in W_{D_L(w)}$, then $xu\in B(w)$. The same is true on right.
}\epf

Write $u\sim_{w}v$ if 
$C_{w}(u)=C_{w}(v)$; thus, 
$\sim_{w}$ defines an equivalent class on $B(w)$. 
In other words, Bruhat cosets of $w$ $\{C_{w}(u)\mid u\in B(w)\}$ give a partition of $B(w)$:
\[
B(w)=\bigcup_{u\in B(w)} C_{w}(u).
\]
But, this is just a set-theoretical result.
We will go into more details of this partition 
with the ideas of graded posets and Bruhat graphs.

\subsection{quotient lower interval}

\begin{defn}
Let $u\in B(w)$. We say that $(u, w)$ is a \eh{critical pair} if 
\[
D_L(w)\subseteq D_L(u) \,\text{   and   }\, D_R(w)\subseteq D_R(u).
\]
(It seems that this terminology ``critical pair" sometimes appears only in the context of \eh{Kazhdan-Lusztig polynomials})
\end{defn}
\begin{ex}
We work on the type $A_{3}$ Coxeter group ($s_{i}$ is the transposition of $i$ and $i+1$).
For example, $(u, w)=(1324, 3412)=(s_{2}, s_{2}s_{1}s_{3}s_{2})$ is a critical pair with 
\[
\dlw=D_{L}(u)=\{s_{2}\}, \q
\drw=D_{R}(u)=\{s_{2}\}.
\]
Another example is $(u, w)=(14325, 45312)=(s_{2}s_{3}s_{2}, 
s_{2}s_{1}s_{3}s_{2}s_{1}s_{4}s_{3}s_{2})$ with 
\[
\dlw=D_{L}(u)=\{s_{2}, s_{3}\},\q 
\drw=D_{R}(u)=\{s_{2}, s_{3}\}.
\]
\end{ex}

\bd{
Let
\begin{align*}
B^\uparrow(w)&=
\{u\in B(w) \mid (u, w) \text{ is a critical pair}\}
,\\
B_\downarrow(w)&=B(w)\cap {}^{D_{L}(w)}W^{D_{R}(w)}.
\end{align*}
}\ed

Note that each $u \in B^\uparrow(w)$ ($B_\downarrow(w)$) is the maximal (minimal) length coset representative of $C_w(u)$.
Note also that 
$w\in B^{\up}(w)$ and $e\in B_{\down}(w)$
so that these sets are nonempty.

\bd{Define projections 
$P^\up:B(w)\to B^\uparrow(w)$ and 
$P_{\down}:B(w)\to B_\down(w)$ by
\[
P^\up(u)=\max C_{w}(u), \q 
P_{\down}(u)=\min C_{w}(u).
\]
(these $P^\up, P_{\down}$ indeed depend on $w$; however, we suppress the letter $w$ for simplicity)
}\ed


\bp{\label{rep}
Both of $P^\up, P_\down$ are surjections and weakly order-preserving in Bruhat order. In other words, in $B(w)$, we have 
\[
u\le v \Longrightarrow P^\up(u)\le P^\up(v) \iff P_\down(u)\le P_\down(v).\]
As a consequence, two subposets $B^\uparrow(w)$ and $B_\downarrow(w)$ in $B(w)$ \eh{are} isomorphic as posets.
}\ep

Let $C(w)=\{C_w(u)\mid u\in B(w)\}$. Whenever no confusion will arise, define projections 
with the same symbols 
$P^\up:C(w)\to B^\uparrow(w)$ and 
$P_{\down}:C(w)\to B_\down(w)$ by
\[
P^\up(C)=\max C, \q P_{\down}(C)=\min C.
\]

\bd{For $C, D\in C(w)$, define 
the \eh{quotient Bruhat order} $C\le D$ if $P^\up(C)\le P^\up(D)$ (equivalently, 
$P_\down(C)\le P_\down(D)$).
Call $(C(w), \le)$ the \emph{quotient lower interval} of $w$. 
}\ed

As this name suggested, $(C(w), \le)$ is indeed an interval, i.e., it has 
both the maximum and minimum.
Observe now that 
\[
(C(w), \le)\cong (B^\uparrow(w),\le) \cong (B_\downarrow(w), \le)
\]
as posets.

\subsection{side, middle length}

\begin{defn}
Let $u\in B(w)$.
The \eh{middle length} of $u$ under $w$ is 
\[
\md_{w}(u)=\el(\pd(u)).
\]
The \eh{side length} of $u$ under $w$ is 
\[
\side_{w}(u)=\el(u)-\el(\pd(u)).
\]
\end{defn}
Observe that $\md_{w}$ is weakly increasing:
\[
u\le v \then \md_{w}(u)\le\md_{w}(v).
\]
(while $\side_{w}$ is increasing only in the same Bruhat coset of $w$:
\[
u\le v, u\sim_{w}v \then \side_{w}(u)\le\side_{w}(v))
\]
\begin{defn}
The \emph{middle length, side length, length} of $C\in C(w)$ are 
\begin{align*}
	\tn{mid}_{w}(C)&=\md_{w}(P_{\down} (C)),
	\\\tn{side}_{w}(C)&=\side_{w}(P^{\up} (C)),
	\\\el(C)&=\tn{mid}_{w}(C)+\tn{side}_{w}(C).
\end{align*}
\end{defn}

Consequently, we also have 
\[
C\le D \then \md_{w}(C)\le\md_{w}(D).
\]
(while 
\[
C\le D \then \side_{w}(C)\le\side_{w}(D)
\]
is not always true)


\begin{ex}
Let us see the Bruhat cosets of $w=3412=s_{2}s_{1}s_{3}s_{2}$ 
with $\dlw=\{s_{2}\}=\drw$ in Figure \ref{3412}; $B(w)$ consists of 14 elements with 
\begin{align*}
	B^\uparrow(w)&=\{3412, 3214, 1432,1324\}
	=\{s_{2}s_{1}s_{3}s_{2}, s_{2}s_{1}s_{2}, s_{2}s_{3}s_{2}, s_{2}\},
	\\B_\downarrow(w)&=\{2143, 2134, 1243, 1234\}
	=\{s_{1}s_{3}, s_{1}, s_{3}, e\}.
\end{align*}

\begin{figure}
\caption{$B(3412), B^{\up}(3412), B_{\down}(3412)$}
\label{3412}
\begin{center}
\resizebox{.95\tw}{!}{
\xymatrix@C=3mm@R=23mm{
&
&&
&&
*+{\fboxrule3pt
\fcolorbox[gray]{0}{0.85}
{3412}}\ar@{-}[dr]\ar@{-}[dlllll]\ar@{-}[drrrrr]
&&
&&
&\\
*+{\fboxrule3pt
\fcolorbox[gray]{0}{0.85}
{3214}}&&
&&3142\ar@{-}[ru]\ar@{-}[dr]
&&
*+{2413}\ar@{-}[dlllll]\ar@{-}[drrr]
&&
&&
*+{\fboxrule3pt
\fcolorbox[gray]{0}{0.85}
{1432}}\\
&*+{{2314}}\ar@{-}[lu]\ar@{-}[drrrr]
&&3124\ar@{-}[lllu]\ar@{-}[ru]\ar@{-}[d]
&&
*+{\fboxrule1.5pt
\fcolorbox[gray]{0}{0.9}
{2143}}\ar@{-}[ru]&&
1342\ar@{-}[lllu]\ar@{-}[rrru]\ar@{-}[d]
&&
*+{{1423}}\ar@{-}[ru]\ar@{-}[dllll]&\\
&&&*+{\fboxrule1.5pt
\fcolorbox[gray]{0}{0.9}
{2134}}\ar@{-}[llu]\ar@{-}[rru]&&
*+{\fboxrule3pt
\fcolorbox[gray]{0}{0.85}
{1324}}\ar@{-}[llu]\ar@{-}[rru]&&
*+{\fboxrule1.5pt
\fcolorbox[gray]{0}{0.9}
{1243}}\ar@{-}[llu]\ar@{-}[rru]&&&\\
&&&&&*+{\fboxrule1.5pt
\fcolorbox[gray]{0}{0.9}
{1234}}\ar@{-}[llu]\ar@{-}[u]\ar@{-}[rru]&&&&&\\
}}
\end{center}
\end{figure}

Both of $B^{\uparrow}(w)$ and $B_{\downarrow}(w)$ are isomorphic to 
the following graded poset of rank 2:
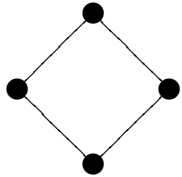
\[
\begin{minipage}[c]{.4\tw}
\begin{xy}
0;<1mm,0mm>:
,(20,10)*{\ci}
,(30,0)*{\ci}
,(30,20)*{\ci}
,(40,10)*{\ci}
\ar@{-}(20,10);(30,0);
\ar@{-}(20,10);(30,20);
\ar@{-}(30,0);(40,10);
\ar@{-}(30,20);(40,10);
\end{xy}
\end{minipage}
\]
\end{ex}

{\renewcommand{\arraystretch}{1.5}
\begin{table}[htp]
\caption{$C(w)$ for $w=3412$}
\label{tab1}
\begin{center}
	\begin{tabular}{cccccccccc}\h
	$C$&$P^{\up}{(C)}$&$\el(C)$	&$P_{\down}(C)$&
	$\md_{w}(C)$&$\side_{w}(C)$\\\h
	$\{3412, 3142, 2413, 2143\}$&3412&4&2143&2&2\\\h
		$\{3214, 2314, 3124, 2134\}$&3214&3&2134&1&2	\\\h
	$\{1432, 1342, 1423, 1243\}$&1432&3&1243&1&2\\\h
	$\{1324, 1234\}$&1324&1&1234&0&1\\\h
\end{tabular}
\end{center}
\end{table}}

Thus, $C(w)$ consists of four cosets; 
three of them are of degree 2 while 
the rest is of degree 1 (note that $\deg(C)=\side_{w}(C)$).
They are all graded posets themselves; 
however, except the bottom coset, they do not contain the minimum element $e$ of $B(w)$; see Table \ref{tab1}.

To describe such details on subposets in a graded poset, we need to introduce some new terminology as below.

%

\subsection{subposets in a graded poset}

Recall that a poset $(P,\le)$ is \eh{graded} if 
there exists some function
$r:P\to\{0,1,2,\ds, n\}$ such that 
\begin{enumerate}
	\item there exists $\wh{0}_{P}=\wh{0}\in P$
	 such that $\wh{0}\le x$ for all $x\in P$,
	\item $r(\wh{0})=0$ and $n=\max_{x\in P} r(x)$,
	\item all maximal chains in $[\wh{0}, x]$
	have the same length.
\end{enumerate}

In this case, we say that 
$(P,\le, r)$ is a graded poset of rank $n$.

We must be very careful when we talk about subposets of $(P, \le, r)$.
Let $Q$ be a nonempty subset of $P$. 
Then, naturally $(Q, \le)$ is a subposet of $(P, \le)$.
It may or may not be graded with the rank function $r$ of $P$, though.

\begin{defn}
Say $(Q,\le)$ is a \eh{faithful subposet} of $(P,\le, r)$ if 
$(Q, \le, r)$ is a graded poset.
\end{defn}
In this case, it is necessary $\wh{0}\in Q$. So a little weaker idea is this:
\begin{defn}
Say $(Q,\le)$ is an \eh{almost faithful subposet} of $(P,\le, r)$ if 
there exists some function $r_{Q}:Q\to\{0, 1, 2, \ds, n_{Q}\}$ such that 
\begin{enumerate}
	\item $(Q, \le, r_{Q})$ is graded itself,
	\item $r_{Q}(x)=r_{P}(x)-n_{Q}\, (x\in Q)$ for some nonnegative integer $n_{Q}$.
\end{enumerate}
\end{defn}
In this case, there exists a unique $\wh{0}_{Q}\in Q$ such that $r_{Q}(\wh{0}_{Q})=0$ and $r_{P}(\wh{0}_{Q})=n_{Q}$.



%

\subsection{main theorem 1}

Now, we summarize all the our results above as a Theorem with remarks.

\begin{thm}
\label{mth1}
Every lower interval $B(w)$ is
 a disjoint union of two-sided cosets 
parametrized by the poset $(C(w), \le)$, the quotient lower interval of $w$,
as \[
B(w)=\bigcup_{C\in C(w)} C
\]
with all of the following:
\begin{itemize}
\item The poset $(C(w), \le)$ is an interval.
\item Each Bruhat coset $C$ of $w$ is a subinterval of $(B(w), \le)$.
 Thus, $(C, \le , \side_{w})$ is an almost faithful subposet of 
rank \,$\side_{w}(C)$. 
In particular, the bottom coset is a faithful subposet of $(B(w), \le, \el)$.
\item As a directed graph, each $C$ is $\side_{w}(C)$-regular.
\item There are nonempty subposets
$(B_{\down}(w), \le)$ and $(B^{\up}(w), \le)$ in $(B(w),\le,\el)$ such that 
\[
B_{\down}(w)=\min\{C\mid C\in C(w)\}, \q 
B^{\up}(w)=\max\{C\mid C\in C(w)\}
\]
and $B_{\down}(w)$ are isomorphic to $B^{\up}(w)$ as posets.
They both parametrize the partition in the following sense:
\[
B(w)=\bigcup_{u\in B_{\down}(w)} C_{w}(u)
=\bigcup_{v\in B^{\up}(w)} C_{w}(v).
\]
\item 
There is a natural partition of $\el$, the rank function of $(B(w), \le)$, by two functions both of which take nonnegative integers:
\[
\el(u)=\md_{w}(u)+\side_{w}(u), \q u\in B(w),
\]
\[
\md_{w}(u)\ge0, \q \side_{w}(u)\ge 0.
\]
Moreover, the function 
$\md_{w}:B(w)\to \{0 ,1, 2, \ds, \md_{w}(w)\}$ is weakly increasing.
\end{itemize}
\end{thm}

\begin{rmk}
\hf
\begin{itemize}
	\item 
	Note that 
$C_{w}(u)=
W_{\dlw}uW_{\drw}$
and 
$C_{w}(v)=
W_{\dlw}vW_{\drw}$
may or may not be isomorphic as posets
even though these two-sided cosets have expressions of the same two-sided indices. This is a phenomenon totally different from one-sided cosets (Fact \ref{cosetpart}).
\item Some of posets $(B_{\down}(w), \le)$ or $(B^{\up}(w), \le)$ may or may not be graded.
\item Let $Q_{1}, Q_{2}$ be subsets of a graded poset $P$. 
We remark that all of the following can happen simultaneously:
	\begin{itemize}
	\item[(1)] $(Q_{1}, \le)$ is a faithful subposet of $(P, \le, r)$ while $(Q_{2},\le)$ is not.
	\item[(2)] $(Q_{2}, \le)$ is graded with some other rank function.
	\item[(3)] $(Q_{1}, \le)$ and $(Q_{2}, \le)$ are isomorphic as posets.
\end{itemize}
We will see examples in the next section with $P=B(w)$, $Q_{1}=B_{\down}(w)$ and $Q_{2}=B^{\up}(w)$ for some $w$.
\end{itemize}
\end{rmk}

%
%

%





\section{Quotient Bruhat graph}

\subsection{separated element}

Next, we study quotient lower intervals in some simplest cases. 
Let 
\[
S(w)=\{s\in S\mid s\le w\}.
\]
\begin{defn}
We say that $w$ is \eh{separated} if 
\[
\dlw\cap S(\pd(w))=
\drw\cap S(\pd(w))=\ku.
\]
\end{defn}
In this case, any $u\in [e, \pd(w)]$ satisfies
\[
s\in \dlw \then \el(su)>\elu
\]
and
\[
s\in \drw \then \el(us)>\elu.
\]
As a result, if $w$ is separated, then 
\[
B_{\down}(w)={}^{\dlw}[e, \pd(w)]^{\drw}=[e, \pd(w)].
\]
\begin{prop}
If $w$ is separated, then 
$(B_{\down}(w), \le, \el)$ is a faithful subposet of $(B(w), \le, \el)$.
\end{prop}
\begin{proof}
$B_{\down}(w)={}^{\dlw}[e, \pd(w)]^{\drw}=[e, \pd(w)]$ 
is just a lower subinterval of $B(w)$ and hence is graded.
\end{proof}

\begin{ex}
On the one hand, all of $w=3412, 45312, 52341$
are separated since 
\begin{align*}
	3412&=s_{2}|s_{1}s_{3}|s_{2},
	\\45312&=s_{2}s_{3}s_{2}|s_{1}s_{4}|s_{2}s_{3}s_{2},
	\\52341&=s_{1}s_{4}|s_{2}s_{3}s_{2}|
s_{1}s_{4}
\end{align*}
here $|\,\cdots\,|$ indicates $P_{\down}(w)$
 for convenience.
On the other hand, $w=456123$ is not separated since 
\[
456123=
s_{3}
|s_{2}s_{1}s_{4}s_{3}s_{2}s_{5}s_{4}|s_{3}
\]
with $\dlw\cap S(
s_{2}s_{1}s_{4}s_{3}s_{2}s_{5}s_{4})=
\{s_{3}\}\ne \ku$.
\end{ex}

We have already seen that 
the function $\md_{w}$ on $B(w)$ is weakly increasing. We can say more with our assumption:
\begin{prop}
If $w$ is separated and $u\le v$ in $B(w)$, then 
\[
\side_{w}(u)\le \side_{w}(v).
\]
\end{prop}

\begin{proof}
Since $(B(w), \le, \el)$ is a graded poset, 
it is enough to prove this for $u\to v$, $\el(u, v)=1$.
Write 
\[
v=x\pd (v)y, \q x\in W_{\dlw}, y\in W_{\drw}
\]
with $\elv=\elx+\el(\pd(v))+\ely$.
\begin{itemize}
	\item Case 1: If $u\simw v$, then $\pd(u)=\pd(v)$ so that 
\[
1=\el(u, v)=\md_{w}(u, v)+\side_{w}(u, v)=0+\side_{w}(u, v)=\side_{w}(u, v).
\]
Thus, 
\[
\side_{w}(u)=\side_{w}(v)-1<\side_{w}(v).
\]
\item Case 2: Suppose $u\not\simw v$.
There exists a reduced expression
\[
v=a_{1}\cd a_{l}
b_{1}\cd b_{m}c_{1}\cd c_{n}
\]
for some $a_{i}, b_{i}, c_{i}\in S$ such that 
$a_{i}\in \dlw, c_{i}\in \drw$ for all $i$ and 
\[
b_{1}\cd b_{m}=\pd(v)
\text{ (reduced)}.
\]
Since $u\to v$ and $u\not\simw v$, we have 
\[
u=a_{1}\cd a_{l}
b_{1}\cd \wh{b_{j}}\cd b_{m}c_{1}\cd c_{n}
\]
for some $j$ (Fact \ref{esub}).
Now, this word must be also reduced 
since $\el(u,v)=1$.
This implies $\pd(u)=b_{1}\cd \wh{b_{j}}\cd b_{m}$
(reduced) so that $\md_{w}(u, v)=1$ and 
\[
\side_{w}(u, v)
=
\el(u, v)-\md_{w}(u, v)=1-1=0.
\]
Thus, $\side_{w}(u)=\side_{w}(v)$.
\end{itemize}
\end{proof}


\begin{ob}
Suppose $w$ is separated.
Then, $(C(w), \le, \md_{w})$ is graded.
In particular, if $C\le D$ in $C(w)$, then 
\[
\self(C)\le \self(D).\]
%
%
%
\end{ob}


\subsection{quotient Bruhat graph}

Now we are ready for introducing quotient Bruhat graphs.
\begin{defn}
The \eh{quotient Bruhat graph} of $w$ is the directed graph with the vertex set $C(w)$ and edges $C\to D$ if 
$C\ne D$ and there exists 
a directed edge $u\to v$ such that 
$u\in C$ and $v\in D$.
\end{defn}
In this way, 
we can regard $C(w)$ not only a poset but also a directed graph.



%
%
%


If $P_{\down}(C)\to P_{\down}(D)$, then $C\to D$ by definition. 
If $w$ is separated, the converse also holds:
\begin{prop}
If $w$ is separated and $C, D \in C(w)$, $C\to D$, then $P_{\down}(C)\to P_{\down}(D)$.
\end{prop}
This follows from the following lemma.


\begin{lem}
Suppose $w$ is separated. If $u\to v$ in $B(w)$ and 
$u\not\simw v$, then $P_{\down}(u)\to P_{\down}(v)$.
\end{lem}

\begin{proof}
Write 
\[
v=a_{1}\cd a_{l}b_{1}\cd b_{m}c_{1}\cd c_{n}
\tn{\q  (reduced)}
\]
with $a_{i}\in \dlw, b_{i}\in S, c_{i}\in \drw$
and 
\[
\pd(v)=b_{1}\cd b_{m}
\,\,\te{ (reduced).}
\]

Since $u\to v$, we can obtain a word for $u$ by deleting one simple reflection from the reduced word for $v$ above; there are three possible cases:
\[
u=a_{1}\cd \wh{a_{j}}\cd a_{l}b_{1}\cd b_{m}c_{1}\cd c_{n},
\]
\[
u=a_{1}\cd  a_{l}b_{1}\cd \wh{b_{j}}\cd b_{m}c_{1}\cd c_{n},
\]
\[
u=a_{1}\cd a_{l}b_{1}\cd b_{m}c_{1}\cd \wh{c_{j}}\cd c_{n}.
\]
In the first and third cases, we must have 
$u\sim_{w}v$,
a contradiction. Thus, it is necessary that 
\[
u=a_{1}\cd  a_{l}b_{1}\cd \wh{b_{j}}\cd b_{m}c_{1}\cd c_{n}
\]
(which may or may not be reduced).
Then we have 
\[
b_{1}\cd \wh{b_{j}}\cd b_{m}<\pd(v)\le \pd(w)
\]
where the first inequality is due to Subword Property.
Since $w$ is separated, this implies
\[
b_{1}\cd \wh{b_{j}}\cd b_{m}\in 
[e, \pd(w)]=B_{\down}(w),
\]
and moreover
\[
\pd(u)=b_{1}\cd \wh{b_{j}}\cd b_{m}.\]
With Fact \ref{esub}, conclude that $\pd(u)\to \pd(v)$.
\end{proof}

Notice that, in general, 
$u\to v, u\not\simw v$
does not necessarily imply $P^{\up}(u)\to P^{\up}(v)$ 
(and $C\to D$ does not necessarily imply 
$P^\up(C)\to P^{\up}(D)$).


\subsection{main theorem 2}

Now, we again summarize the our results above as a Theorem with remarks.

\begin{thm}
\label{mth2}
Suppose $w$ is separated. Then,
every lower interval $B(w)$ is
 a disjoint union of two-sided cosets
 paramtrized by its quotient lower interval
\[
B(w)=\bigcup_{C\in C(w)} C
\]
with all of the following:
\begin{itemize}
	\item $B_{\down}(w)$ are isomorphic to $B^{\up}(w)$ as posets.
\item $(B_{\down}(w), \le)$ is a faithful subposet of $(B(w),\le, \el)$.
\item We have
\[
B(w)=\bigcup_{u\in B_{\down}(w)} C_{w}(u)
=\bigcup_{v\in B^{\up}(w)} C_{w}(v).
\]
\item 
There is a natural partition of $\el$, the rank function of $(B(w), \le)$, by two functions both of which take nonnegative integers:
\[
\el(u)=\md_{w}(u)+\side_{w}(u), \q u\in B(w),
\]
\[
\md_{w}(u)\ge0, \q \side_{w}(u)\ge 0.
\]
Moreover, both of $\md_{w}, \side_{w}$ are weakly increasing.
\item Each Bruhat coset $(C, \le)$ of $w$ is a subinterval of $(B(w), \le)$ and hence $(C, \le, \side_{w})$ is an almost faithful subposet of 
\rank\, $\side_{w}(C)$
and, as a directed graph, it is $\side_{w}(C)$-regular.
Moreover, the degree of such cosets is weakly increasing 
in quotient Bruhat order:
\[
C\le D \then \deg_{}(C)\le \deg_{}(D).
\]
\item 
There is an equivalence 
\[
C\to D \iff P_{\down}(C)\iff P_{\down}(D)
\]
so that $(C(w), \to)\cong (B_{\down}(w), \to)$ as a directed graph.
\end{itemize}
\end{thm}

\begin{rmk}
We should not misunderstand this theorem as 
the \eh{direct product of graded posets}.
Even though $(B(w), \le, \el)$ has two kinds of graded poset structures 
$(B_{\down}(w), \le, \md_{w})$
and $(C, \le, \side_{w})$ inside with 
$\el(u)=\md_{w}(u)+\side_{w}(u)$, $B(w)$ is \eh{not} necessarily the direct product of them.
\end{rmk}

\subsection{example 1: 45312}
Here we observe some examples from type A Coxeter groups.

%
%
%
%
%
%
%
Let $w=45312=s_{2}s_{3}s_{2}|s_{1}s_{4}|s_{2}s_{3}s_{2}$  (separated).
We have 
\begin{align*}
	B^{\up}(w)&=\{45312, 43215, 15432, 14325\},
	\\B_{\down}(w)&=\{21354, 21345, 12354, 12345\}.
\end{align*}
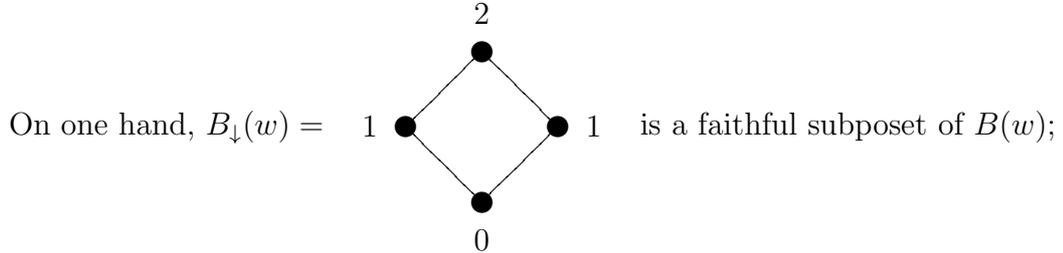
\begin{center}
On one hand, $B_{\down}(w)= $
\begin{minipage}[c]{.4\tw}
\begin{xy}
0;<1mm,0mm>:
,(20,10)*++{\ci}*+++!R{1}
,(30,0)*++{\ci}*+++!U{0}
,(30,20)*++{\ci}*+++!D{2}
,(40,10)*++{\ci}*+++!L{1}
\ar@{-}(20,10);(30,0);
\ar@{-}(20,10);(30,20);
\ar@{-}(30,0);(40,10);
\ar@{-}(30,20);(40,10);
\end{xy}
\end{minipage} is a faithful subposet of $B(w)$;
Those numbers indicate length in $B(w)$.

On the other hand,
$B^{\up}(w)= $
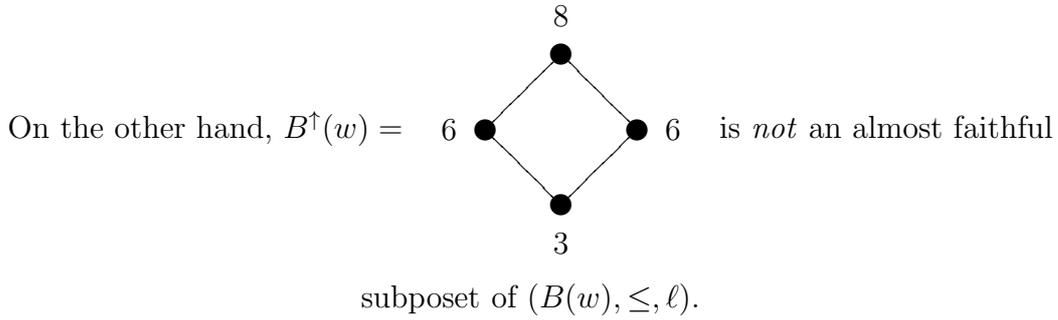
\begin{minipage}[c]{.4\tw}
\begin{xy}
0;<1mm,0mm>:
,(20,10)*++{\ci}*+++!R{\fn 6}
,(30,0)*++{\ci}*+++!U{\fn 3}
,(30,20)*++{\ci}*+++!D{\fn 8}
,(40,10)*++{\ci}*+++!L{\fn 6}
\ar@{-}(20,10);(30,0);
\ar@{-}(20,10);(30,20);
\ar@{-}(30,0);(40,10);
\ar@{-}(30,20);(40,10);
\end{xy}

\end{minipage}
is \eh{not} an almost faithful subposet of $(B(w), \le, \el)$.
\end{center}
The quotient Bruhat graph $C(w)$ consists of four cosets; degree of those is 3, 5, 5 and 6, respectively.
\begin{center}
$C(w)\cong [e, s_{1}s_{4}]=$
\begin{minipage}[c]{.4\tw}
\begin{xy}
0;<5mm,0mm>:
,0+(-12,1)*{\ci}="b1"*+++!U{}
,(-3,3)+(-12,1)*{\ci}="al1"*++!R{}
,(3,3)+(-12,1)*{\ci}="ar1"*++!L{}
,(0,6)+(-12,1)*{\ci}="c"*++!R{}
,\ar@{->}"b1";"al1"
,\ar@{->}"b1";"ar1"
,\ar@{->}"al1";"c"
,\ar@{->}"ar1";"c"
\end{xy}
\end{minipage}.

%

{\renewcommand{\arraystretch}{1.5}
\begin{table}[h!]
\caption{$C(w)$ for $w=45312$}
\begin{center}
	\begin{tabular}{lcccccccccc}\h
	$P^\up{(C)}$&$P_\down{(C)}$	&$\el(C)$&$\md_{w}(C)$&$\self(C)$	\\\h
	$45312=s_{2}s_{3}s_{2}|s_{1}s_{4}|s_{2}s_{3}s_{2}$&$s_{1}s_{4}$	&8&2&6
	\\\h
	$43215=s_{2}s_{3}s_{2}|s_{1}|s_{2}s_{3}$&$s_{1}$		
		&6&1&5
	\\\h
	$15432=s_{2}s_{3}s_{2}|s_{4}|s_{3}s_{2}$&$s_{4}$		
		&6	&1&5
	\\\h
	$14325=s_{2}s_{3}s_{2}|e|e$&$e$
				&3&0&3
	\\\h
\end{tabular}
\end{center}
\end{table}%
}

\end{center}

\subsection{example 2: 52341}

Let $w=52341=s_{1}s_{4}|s_{2}s_{3}s_{2}|s_{1}s_{4}$ (Table \ref{tab3}).
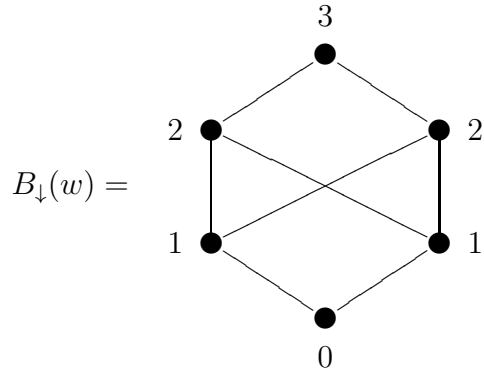
\begin{center}
$B_{\down}(w)=$
\begin{minipage}[c]{.4\tw}
\begin{xy}
0;<5mm,0mm>:
,0+(-12,0)*{\ci}="b1"*+++!U{0}
,(-3,2)+(-12,0)*{\ci}="al1"*+++!R{1}
,(3,2)+(-12,0)*{\ci}="ar1"*+++!L{1}
,(-3,5)+(-12,0)*{\ci}="cl1"*+++!R{2}
,(3,5)+(-12,0)*{\ci}="cr1"*+++!L{2}
,(0,7)+(-12,0)*{\ci}="t1"*+++!D{3}
,\ar@{-}"b1";"al1"
,\ar@{-}"b1";"ar1"
,\ar@{-}"al1";"cl1"
,\ar@{-}"ar1";"cr1"
,\ar@{-}"cl1";"t1"
,\ar@{-}"cr1";"t1"
,\ar@{-}"al1";"cr1"
,\ar@{-}"ar1";"cl1"
\end{xy}
\end{minipage}
\end{center}
\begin{center}
$B^{\up}(w)=$
\begin{minipage}[c]{.4\tw}
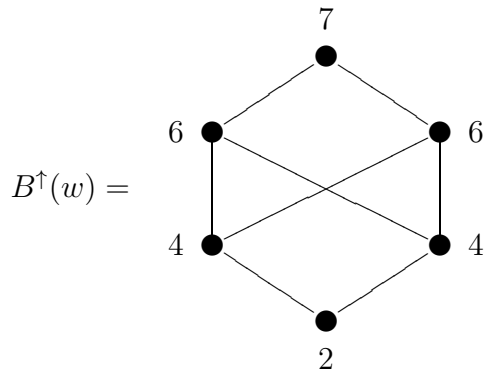
\begin{xy}
0;<5mm,0mm>:
,0+(-12,0)*{\ci}="b1"*+++!U{2}
,(-3,2)+(-12,0)*{\ci}="al1"*+++!R{4}
,(3,2)+(-12,0)*{\ci}="ar1"*+++!L{4}
,(-3,5)+(-12,0)*{\ci}="cl1"*+++!R{6}
,(3,5)+(-12,0)*{\ci}="cr1"*+++!L{6}
,(0,7)+(-12,0)*{\ci}="t1"*+++!D{7}
,\ar@{-}"b1";"al1"
,\ar@{-}"b1";"ar1"
,\ar@{-}"al1";"cl1"
,\ar@{-}"ar1";"cr1"
,\ar@{-}"cl1";"t1"
,\ar@{-}"cr1";"t1"
,\ar@{-}"al1";"cr1"
,\ar@{-}"ar1";"cl1"
\end{xy}
\end{minipage}\end{center}

\begin{center}
$C(w)\cong [e, s_{2}s_{3}s_{2}]$=
\begin{minipage}[c]{.4\tw}
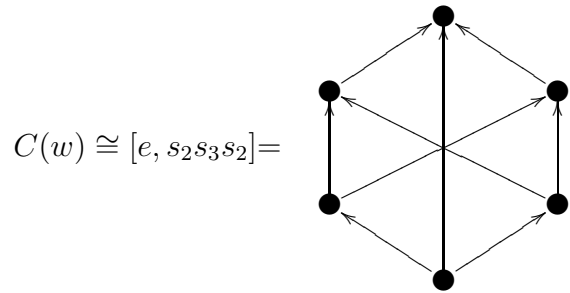
\begin{xy}
0;<5mm,0mm>:
,0+(-12,0)*{\ci}="b1"*+++!U{}
,(-3,2)+(-12,0)*{\ci}="al1"*++!R{}
,(3,2)+(-12,0)*{\ci}="ar1"*++!L{}
,(-3,5)+(-12,0)*{\ci}="cl1"*++!R{}
,(3,5)+(-12,0)*{\ci}="cr1"*++!L{}
,(0,7)+(-12,0)*{\ci}="t1"*+++!D{}
,\ar@{->}"b1";"al1"
,\ar@{->}"b1";"ar1"
,\ar@{->}"al1";"cl1"
,\ar@{->}"ar1";"cr1"
,\ar@{->}"cl1";"t1"
,\ar@{->}"cr1";"t1"
,\ar@{->}"al1";"cr1"
,\ar@{->}"ar1";"cl1"
,\ar@{->}"b1";"t1"
\end{xy}
\end{minipage}
\end{center}

%


{\renewcommand{\arraystretch}{1.5}
\begin{table}[htp]
\caption{$C(w)$ for $w=52341$}
\label{tab3}
\begin{center}
	\begin{tabular}{lccccccccccc}\h
	$P^{\up}(C)$&$P_{\down}(C)$	&$\el(C)$&$\md_{w}(C)$&$\self(C)$\\\h
	$52341=s_{1}s_{4}|s_{2}s_{3}s_{2}|s_{1}s_{4}$&$s_{2}s_{3}s_{2}$		&7&
	3&4
	\\\h
	$52143=s_{1}s_{4}|s_{3}s_{2}|s_{1}s_{4}$&$s_{3}s_{2}$		&6&
	2&4
	\\\h
	$32541=s_{1}s_{4}|s_{2}s_{3}|s_{1}s_{4}$&$s_{2}s_{3}$		&6&
	2&4
	\\\h
	$21543=s_{1}s_{4}|s_{3}|s_{4}$&$s_{3}$		&4
		&1&3
	\\\h
	$32154=s_{1}|s_{2}|s_{1}s_{4}$&$s_{2}$		&4
		&1&3	
	\\\h
	$21354=s_{1}|e|s_{4}$&$e$		&2
			&0&2
	\\\h
\end{tabular}
\end{center}
\end{table}%
}



%


%


\section{Future research}
We end with recording some ideas for our future research.

\begin{enumerate}
	\item 
	Let $\deg_{w}(u)$ be the degree of $u$ in $B(w)$ 
	(graph-theoretic degree = the number of edges incident to $u$).
	We can show that if $u\simw v$, then 
	$\deg_{w}(u)=\deg_{w}(v)$.
	Moreover, we always have $\deg_{w}(u)\ge \elw$ (known as 
	\eh{Deodhar inequality} \cite{billey}).
	Can we prove
	\[
u\le v \mb{ in $B(w)$ } \then
	\deg_{w}(u)\ge \deg_{w}(v) \tn{\mb{\ph{a}}}?\]
	\item What if $w$ is not separated? 
	What more can you say about (poset, graded poset, directed graph) structures of 
	$B^{\up}(w), B_{\down}(w), C(w)$ in $(B(w), \le, \el)$?
	\item Study the \eh{Poincare polynomial} of $B(w)$
	\[
\P_{w}(q)=\sum_{u\in B(w)}q^{\elu}\]
	from an aspect of quotient lower intervals:
	for example, find a relation between $\P_{w}(q)$ and $\P_{\pd(w)}(q)$.
\end{enumerate}

\appendix

\section{Basics on Coxeter groups}

\subsection{reduced words, weak orders, subword property}\label{i3}
Let $W=(W, S)$ be a Coxeter system.
By $T$, we mean the set of its reflections:
\[
T=\Set{w^{-1}sw}{w\inw, s\ins}.
\]

Since $S$ is a group-theoretic generator of $W$, 
for each $w\inw$, we have 
\[
w=s_{1}\cd s_{n}
\]
for some $s_{1}, \ds, s_{n}\in S$.
\begin{defn}
The \eh{length} of $w$ is 
\[
\el(w)=\min\{l\ge 0\mid 
w=s_{1}\cd s_{l}\}.
\]
\end{defn}
We say the word $w=s_{1}\cd s_{l}$ is \eh{reduced} if 
$l=\elw$.
\begin{defn}
\[D_{L}(w)=\{s\in S\mid \el(sw)<\elw\},\]
\[D_{R}(w)=\{s\in S\mid \el(ws)<\elw\}.\]
\end{defn}
\begin{fact}
\cite[Lifting Property, Proposition 2.2.7]{bb}
Let $u<w$. If $s\in D_L(w)\setminus D_L(u)$, then $su\le w.$
If $s\in D_R(w)\setminus D_R(u)$, then $us\le w.$ 
\end{fact}

\begin{defn}
Write $u\to v$ if $v=tu$ for some $t\in T$ and $\elu<\elv$.
\end{defn}

\begin{defn}
We say that $t\in T$ is a \eh{left inversion} for $w$ if 
$tw\to w$. Similarly, 
$t$ is a \eh{right inversion} for $w$ if $wt\to w$.
Define
\begin{align*}
T_{L}(w)=\{t\in T\mid tw\to w\} \text{ \q and \q}
T_{R}(w)=\{t\in T\mid wt\to w\}.
\end{align*}
\end{defn}

\eh{Left, right weak orders} on $W$ are simply 
defined as follows:
\[
u\le_{L}w \iff T_{R}(u)\sub T_{R}(w),
\]
\[
u\le_{R}w \iff T_{L}(u)\sub T_{L}(w).
\]

From each reduced word for $w$, we can construct 
its inversions as follows:
Let $w=s_1s_2\cdots s_l$ be any reduced word. 
Set $t_i=s_1s_2\cdots s_{i-1}s_is_{i-1} \cdots s_2s_1$ for each $i=1 ,2, \cdots, l$. Then:
\bft{We have $T_{L}(w)=\{t_1, t_2, \dots, t_l\}$ and $t_i$ are all distinct
 (independent from the choice of a reduced word for $w$). Consequently, $l=|T_{L}(w)|=\ell(w)$.
}\eft

\begin{fact}
\label{esub}
Let $w=s_{1}\cd s_{l}$ be a reduced word and $t\in T$.
Then, the following are equivalent:
	\begin{enumerate}
		\item $t\in T_{L}(w)$.
		\item $t= s_{1}s_{2}\cd s_{i-1}s_{i}s_{i-1}\cd s_{2}s_{1}$
		 for some $i$.
		\item $tw= s_{1}\cd \wh{s_{i}}\cd s_{l}$
		for some $i$.	\end{enumerate}
Moreover, the number $i$ the above is unique.
\end{fact}

Notice that 
\[
tw= s_{1}\cd \wh{s_{i}}\cd s_{l}
\]
is a subword of the reduced word 
$s_{1}s_{2}\cd s_{l}$ for $w$.
This may or may not be reduced, though.

\bft[Subword Property]{
Let $w=s_1\cdots s_l$ be an arbitrary reduced word for $w$.
Then, the following are equivalent:
\be{\item $u\le w$ in Bruhat order.
\item There exists a subword of $s_1\cdots s_l$ for $u$; that is, 
\[
u=s_{i_1}\cdots s_{i_k}, \q 1\le i_{1}<\cd <i_{k}\le l.\]
(this word need not be reduced)
}\ee
}\eft

\subsection{one-sided cosets and quotients}

Here, we review some well-known facts. 

\bd{The \emph{standard parabolic subgroup} $W_I$ is the group-theoretic  subgroup of $W$ generated by $I$. 
}\ed

The subgroup $W_{I}$ is a Coxeter system itself with the canonical generators $I$, reflections $T(I)\defeq T\cap W_I$, Bruhat suborder $\le$ and length function $\ell_I$ with respect to $I$.

\bd{Let $V \sub W$.
Say $V$ is a \emph{left coset} if 
\[
V=W_Iu\overset{\tn{def}}{=}\{xu \mid x\in W_I\}\]
 for some $I\sub S$ and $u\in W$.
Similarly, it is a \eh{right coset} if 
\[
V=uW_I\overset{\tn{def}}{=}\{uy \mid y\in W_I\}\]
for some $I\sub S$ and $u\in W$.
It is a \eh{one-sided coset} if it is a left or right coset.
}\ed

\begin{rmk}
%
Note that the choice of index $I$ for each one-sided coset is unique:
\[
W_{I}u= W_{J}u\iff I=J,
\]
\[
uW_{I}= uW_{J}\iff I=J.
\]
\end{rmk}
\bft{
For any $W_Iu$, 
there exist $v_{0}, v_{1}\in W_Iu$,
its minimal and maximal length coset representatives, such that 
\[
\ell(v_0)\le \ell(v)\le \ell(v_1)
\]
for all $v\in W_{I}u$. Indeed, $v_{0}=\min W_{I}u$ and 
$v_{1}=\max W_{I}u$ (here min and max are taken in Bruhat order).
The same is true for right cosets.
}\eft

Next, let us review classical results on \eh{one-sided quotients}.
Define subsets of $W$ for $I, J\subs$ as 
\begin{align*}
\iW&=\{w\inw\mid s\in I\then \el(sw)>\elw\},\\	
W^{J}&=\{w\inw\mid s\in J\then \el(ws)>\elw\}.
\end{align*}
Call these \eh{left and right quotients}.

\begin{fact}\label{coset}
Let $J\sub S$. For each $w\in W$, 
there exists a \eh{unique} pair $
(w_{J}, w^{J})\in W_{J}\ti W^{J}$ such that 
\begin{align*}
	w&=w_{J}\cdot w^{J},
	\\\ell(w)&=\ell(w_{J})+\ell(w^{J}).
\end{align*}
\end{fact}

\begin{fact}[Chain property]
Every one-sided quotient is graded with the rank function $\el$:
If $u<w$ in $\iW$ ($W^{J}$), 
there exists a chain in $\iW$ ($W^{J}$)
\[
u=v_{0}< v_{1}< v_{2}< \cd < v_{k}=w
\]
such that 
$\el(v_{i}, v_{i+1})=1$ for all $i$ (thus $k=\eluw$).
Consequently, 
${}^{I}[u, w]^{}\defeq[u, w]\cap {}^{I}W^{}$
$\k{{}^{}[u, w]^{J}\defeq[u, w]\cap {}^{}W^{J}}$
 is graded.
\end{fact}

\begin{fact}\label{cosetpart}
Every lower interval $B(w)$ is a disjoint union of 
left cosets with the same index $\dlw$:
\[
B(w)=\bigcup_{u\le w} W_{D_{L}(w)} u 
\]
Moreover, all cosets in 
\[
\{W_{D_{L}(w)} u \mid u\le w\}
\]
are isomorphic to each other as posets
 (and even as graded posets and as directed graphs).
Moreover, each $W_{D_{L}(w)} u $ is indeed an (left weak) interval 
and a regular graph of the \eh{same} degree.
\end{fact}

\end{document}